\documentclass[12pt]{article}

\usepackage{amsfonts}
\usepackage{epsfig}
\usepackage{graphicx}
\usepackage{amsmath}
\usepackage{amssymb}

\newcounter{main}

\newtheorem{theorem}{Theorem}[section]
\newtheorem{proposition}[theorem]{Proposition}
\newtheorem{lemma}[theorem]{Lemma}
\newtheorem{corollary}[theorem]{Corollary}

\newtheorem{remark}{Remark}[section]

\newcommand{\blanksquare}{\,\,\,$\sqcup\!\!\!\!\sqcap$}
\newenvironment{proof}{{\flushleft {\bf Proof: }}}{\blanksquare}

\newcounter{example}
{{\stepcounter{example}}{\flushleft {\bf Example \arabic{example}:}}}%
{\par}

\title{\textbf{On $C^1$-robust transitivity of volume-preserving flows}}

\author{ M\'{a}rio Bessa \thanks{Supported by FCT-FSE, SFRH/BPD/20890/2004. }
\space and Jorge Rocha \thanks{Partially supported by FCT-POCTI/MAT/61237/2004.}}

\begin{document}
\maketitle

\begin{abstract}
We prove that a  divergence-free and $C^1$-robustly transitive vector field has no singularities. Moreover, if the vector field is $C^4$ then the linear Poincar\'e flow associated to it admits a dominated splitting over $M$.

\end{abstract}

\bigskip

\emph{keywords:} Volume-preserving flows; Robust transitivity; Dominated splitting; Ergodicity.\\

\begin{section}{Introduction and statement of the\\ results}

It is well known that, in the $C^1$-topology, robust transitivity of a dyna\-mi\-cal system defined on a compact manifold always implies some form of (weak) hyperbolicity. In fact in the early 1980s Ma\~n\'e (\cite{M4}) proved that a $C^1$-robustly transitive two-dimensional diffeomorphism is uniformly hyperbolic. Ma\~n\'e's Theorem was generalized first   by D\'\i az, Pujals and Ures (\cite{DPU}) showing that $C^1$-robustly transitive three-dimensional diffeomorphism are partially hyperbolic, and then by Bonatti, D\'\i az and Pujals (\cite{BDP}) obtaining that a $C^1$-robustly transitive  diffeomorphism has dominated splitting. In the symplectomorphism case Horita and Tahzibi (\cite{HT}) showed that $C^1$-robust transitivity implies partial hyperbolicity in any dimension.

Concerning the vector field context Doering (\cite{D}) transposed Ma\~n\'e's result to three-dimensional flows. Then,   generalizing this result, Vivier (\cite{V}) showed that, in any dimension, $C^1$-robustly transitive vector fields do not have singularities, and Bonatti, Gourmelon and Vivier (\cite{BGV}) proved that the linear Poincar\'e flow of a $C^1$-robustly transitive vector field admits a dominated splitting. In the three-dimensional and volume-preserving case, Arbieto and Matheus (\cite{AM}) showed that a $C^1$-robustly transitive vector field is Anosov. Finally, Vivier (\cite{V2}) proved that any Hamiltonian vector field defined on a four-dimensional  sympletic manifold and admitting a robustly transitive regular energy surface is hyperbolic on this energy surface. 

In this paper we consider the conservative flows setting (or, equivalently, the divergence-free vector fields scenario)  and  obtain the same kind of results of Vivier and of Bonatti, Gourmelon and Vivier mentioned above. Concerning the ergodic theoretical point of view we mention that, using the Ma\~n\'e, Bochi and Viana strategies (\cite{M2} and~\cite{BV}), in ~\cite{B1} is proved  that generically conservative linear differential systems have, for almost every point, zero Lyapunov exponents or else a dominated splitting.

Before stating precisely our results let us introduce some definitions.

\bigskip

Let $M$ be a compact, connected and boundaryless smooth Riemannian manifold of dimension $n \geq 4$. We denote by $\mu$ the Lebesgue measure induced by the Riemannian volume form on $M$. We say that a vector field $X$ is \emph{divergence-free} if its divergence is equal to zero or equivalently if the measure $\mu$ is invariant for the associated flow, $X^t$, $t \in \mathbb{R}$. In this case we say that the flow is \emph{conservative} or \emph{volume-preserving}.

We denote by $\mathfrak{X}_\mu^r(M)$ ($r\geq 1$) the space of $C^r$ divergence-free vector fields of $M$ and endow this set with the usual $C^1$-topology. 

A vector field $X$ is said to be \emph{transitive} if its flow has a dense orbit in $M$. Moreover, $X$ is $C^1$-\emph{robustly transitive} if there exists a $C^1$-neighbourhood of $X$ in $\mathfrak{X}_\mu^1(M)$ such that all its elements are transitive.

Let us now state our first result.

\begin{theorem}\label{T2}
Let $X \in \mathfrak{X}_\mu^1(M)$ be a $C^1$-robustly transitive vector field. Then $X$ has no singularities. 
\end{theorem}

We denote by $Sing(X)$ the set of \emph{singularities} of $X$ and by $R:=M\setminus Sing(X)$ the set of \emph{regular} points. 
Given $x\in R$ we consider its normal bundle $N_{x}=X(x)^{\perp}\subset T_{x}M$ and define the \emph{linear Poincar\'{e} flow} by $P_{X}^{t}(x):=\Pi_{X^{t}(x)}\circ DX^{t}_{x}$ where $\Pi_{X^{t}(x)}:T_{X^{t}(x)}M\rightarrow N_{X^{t}(x)}$ is the projection along the direction of $X(X^{t}(x))$. Let $\Lambda \subset R$ be an $X^t$-invariant set and $N=N^{1}\oplus  N^{2}\oplus ... \oplus N^{k}$ be a $P_{X}^{t}$-invariant splitting over $\Lambda$ such that all the subbundles have  constant dimension. We say that this splitting is
 an $\ell$-\emph{dominated splitting} for the linear Poincar\'{e} flow if there exists an $\ell \in {\mathbb N}$ such that, for all $0 \leq i < j \leq k$ and for all $x\in \Lambda$  we have:
$$\|P_{X}^{\ell}(x)|_{N^{j}_{x}}\|  . \|(P_{X}^{\ell}(x)|_{N^{i}_{x}})^{-1}\|   \leq \frac{1}{2}.$$

Previous result guaranties that a $C^1$-robustly transitive vector field has no singularities. Moreover, next theorem shows that if these vector fields are of class $C^4$ then they exhibit some type of weak hyperbolicity.

\begin{theorem}\label{T1}
Let $X \in \mathfrak{X}_\mu^4(M)$ be a $C^1$-robustly transitive vector field. Then $X$ has no singularities and the linear Poincar\'e flow of  $X$ admits a dominated splitting over $M$. 
\end{theorem}

We point out that this theorem requires that the $C^1$-robustly transitive vector field $X$ is of class $C^4$ and not just of class $C^1$ as it would be expected. This hypothesis is a  technical assumption  needed to make $C^1$-conservative perturbations of the initial vector field $X$.  Actually, the proofs of  Theorems~\ref{T2} and~\ref{T1} are based on the ones made by Vivier (\cite{V}) and by Bonatti, Gourmelon and Vivier (\cite{BGV}) but, as the perturbations are made in the conservative class, we need to develop some  appropriate $C^1$ perturbation lemmas, namely a kind of conservative Franks' lemma, and for that we need to begin with $C^4$ regularity. We also refer that one of the main tools to get these perturbation lemmas is the Arbieto and Matheus Pasting Lemma (\cite{AM}). 

Let us now state some corollaries of the two theorems above.

As, for $r \geq 2$, $C^r$ divergence-free vector fields are $C^1$-dense in $\mathfrak{X}_\mu^1(M)$ (Zuppa,~\cite{Z}) we obtain the following corollary.

\begin{corollary}
Let $\mathcal{T}$ be the open set of $C^1$-robustly transitive vector fields $X \in \mathfrak{X}_\mu^1(M)$.
There exists a $C^1$-dense subset $\mathcal{U}$ of $\mathcal{T}$ such that if $X \in \mathcal{U}$ then $\text{Sing}(X)=\emptyset$  and the linear Poincar\'e flow of  $X$ admits a dominated splitting over $M$.
\end{corollary}

Given a probability measure $\nu$ invariant for the flow $X^t$ we say that $\nu$ is an \emph{ergodic} measure for $X$ if any measurable set that is  invariant by the flow has zero or full measure. Equivalently, $\nu$  is ergodic if, for every observable continuous function $\varphi\colon M \rightarrow \mathbb{R}$ and for $\nu$-a.e. point $x \in M$, one has
$$\int_M \varphi(z) d\nu(z)= \underset{t\rightarrow \infty}{\lim} \frac{1}{t}\int_0^t \varphi(X^s(x))ds.$$
We say that  a $C^r$ vector field $X$ is $C^1$-\emph{stably ergodic} with respect to a probability measure $\nu$, $r \geq 1$, if there exists a $C^1$-neighbourhood of $X$, $\mathcal{U}$, such that  $\nu$ is an ergodic measure for $Y$, for all $Y \in \mathcal{U}$. In this paper we only consider ergodicity and stable ergodicity for the Lebesgue measure.

It is well known that,  for conservative systems, ergodicity implies transitivity. Therefore, as an immediate consequence of Theorem~\ref{T1}, we get the following corollary.

\begin{corollary}
Let $X \in \mathfrak{X}_\mu^4(M)$ be a $C^1$-stably ergodic vector field. Then $\text{Sing}(X)=\emptyset$  and the linear Poincar\'e flow of  $X$ admits a dominated splitting over $M$.
\end{corollary}

Let $X \in \mathfrak{X}_\mu^1(M)$ be a  vector field without singularities and let $E^1 \oplus E^2 ... \oplus E^k=N$ be a dominated splitting for the linear Poincar\'e flow associated to $X$. We consider the integrated Jacobian of  $X$ restricted to each $E^i$ defined by $$\Sigma^i(X)=\int_M \log |\det P_X^1(x)|_{E^i} |d\mu(x),\,\,i \in \{1,2,...,k\}.$$
In~\cite{BR} we prove that a $C^1$-stably ergodic vector field $X$, such that $M \setminus Sing(X)$ is partially hyperbolic and all the singularities are \emph{linear hyperbolic}, can be $C^1$-approximated by a $C^2$-divergence-free vector field $Y$ such that $\Sigma^c(Y)\not=0$, where $\Sigma^c(Y)$ denotes the integrated Jacobian of  $Y$ restricted to the central subbundle $E^c$. Actually, Theorem~\ref{T2} implies that a $C^1$-stably ergodic vector field does not have singularities. We also remark that the proof given on \cite{BR} only requires the existence of a dominated splitting. This fact was already observed in \cite{BB} in the diffeomorphism context. Hence, Theorem 1 of ~\cite{BR} can be reformulated as follows.

\begin{corollary}\label{nbr}
Let $X \in \mathfrak{X}^{1}_{\mu}(M)$  be a stably ergodic flow and $E^1 \oplus E^2 ... \oplus E^k=N$ be a dominated splitting for the linear Poincar\'e flow of $X$ over $M$. Then $X$ may be approximated, in the $C^{1}$-topology, by $Y \in \mathfrak{X}^{\infty}_{\mu}(M)$ for which $\Sigma^i(Y) \not=0$, for all $i \in \{1,2,...,k\}$.
\end{corollary}

Note that if $X \in \mathfrak{X}^{4}_{\mu}(M)$ then the previous result applies directly to the dominated splitting given by Theorem~\ref{T1}.

We recall that a conservative vector field $X$ is said to be \emph{nonuniformly hyperbolic} if all  the Lyapunov exponents are a.e. different from zero. In particular if $M$ is four dimensional,  $X \in \mathfrak{X}_\mu^4(M)$ is $C^1$-stably ergodic and admits a  dominated splitting  with three nontrivial subbundles then the previous corollary assures that $X$ can be $C^1$-approximated by a nonuniformly hyperbolic vector field.

\bigskip

If $M$ is a four dimensional manifold and
$X \in \mathfrak{X}_\mu^4(M)$ is stably ergodic then $M$ admits a dominated splitting, $E \oplus F$, for the linear Poincar\'e flow associated to $X$. Since the vector field is divergence-free and has no singularities it is straightforward to see that the one-dimensional invariant subbundle is hyperbolic and the other subbundle is hyperbolic in volume, that is the splitting is a partially hyperbolic one. Therefore, using again Zuppa's Theorem (\cite{Z}), we obtain the following result.

\begin{corollary}
A $C^1$-stably ergodic vector field $X \in \mathfrak{X}_\mu^1(M^4)$ can be $C^1$-approximated by a partially hyperbolic vector field.
\end{corollary}

This paper is organized as follows. In section~\ref{tour} we state three results (Proposition~\ref{V1},  Proposition~\ref{V2} and Proposition~\ref{VF}) and deduce the theorems from the first two. Proposition~\ref{V1} is an easy adaptation of Proposition 4.1 of Vivier and we show that Proposition~\ref{VF} implies Proposition~\ref{V2}. In section ~\ref{pl} we obtain the perturbation lemmas needed to prove Proposition~\ref{V1} and Proposition~\ref{VF} and, finally,  in section~\ref{MG} we prove Proposition~\ref{VF}.

\end{section}

\begin{section}{A \emph{tour} on the proofs of the theorems}\label{tour} 
 
In this section we prove Theorems~\ref{T2} and~\ref{T1} following the strategy used by Vivier in ~\cite{V} and  by Bonatti, Gourmelon and Vivier in ~\cite{BGV} and adapting some of their results to the conservative setting.
For that we begin by stating, in the divergence-free vector fields scenario, two main results of the references above.

\begin{proposition}\label{V1}
Let $X \in \mathfrak{X}_\mu^1(M)$ be a robustly transitive vector field such that $Sing(X)\not= \emptyset$. Then there exists an arbitrarily $C^1$-close vector field $Y \in \mathfrak{X}_\mu^4(M)$ such that the linear Poincar\'e flow of $Y$ does not admit any dominated splitting over $M\setminus Sing(Y)$.
\end{proposition}

\begin{proposition}\label{V2}
Let $X \in \mathfrak{X}_\mu^4(M)$ be a robustly transitive vector field. Then there exists a dominated splitting for the linear Poincar\'e flow of $X$ over $M\setminus Sing(X)$.

\end{proposition}
Let us now explain how we derive Theorems~\ref{T2} and~\ref{T1} from the previous propositions. 

Fix a robustly transitive vector field $X \in \mathfrak{X}_\mu^1(M)$; let $\mathcal{U}$ be a $C^1$ open neigbourhood of $X$ such that all $Y \in \mathcal{U}$ are transitive, hence robustly transitive. 

Let us assume that $Sing(X)\not=\emptyset$. Applying Lemma~\ref{linear}, we obtain $Y \in \mathfrak{X}_\mu^{\infty}(M) \cap \mathcal{U}$ such that $Sing(Y)\not= \emptyset$ and $Y$ has at least �ne hyperbolic singularity. Therefore, by Proposition~\ref{V1}, we obtain a  vector field $Z \in \mathfrak{X}_\mu^4(M) \cap \mathcal{U}$ such that $Sing(Z)\not= \emptyset$ and the linear Poincar\'e flow of $Z$ does not admit any dominated splitting  over $M\setminus Sing(Z)$, which is in contradiction with Proposition~\ref{V2} applied to $Z$. Therefore $Sing(X)=\emptyset$, which proves Theorem~\ref{T2}.

Now let $X \in \mathfrak{X}_\mu^4(M)$ be a $C^1$-robustly transitive vector field; the previous argument shows that $Sing(X)=\emptyset$ and then Proposition~\ref{V2} guarantees that $M$ admits a dominated splitting for the linear Poincar\'e flow of $X$, thus proving Theorem~\ref{T1}.

\bigskip

Proposition~\ref{V1}, up to a minor detail, is a consequence of Proposition 4.1 of~\cite{V}. To see this let us first recall that a  singularity $p$ of a given vector field $X$ is said to be a \emph{linear hyperbolic singularity} if it is a hyperbolic singularity and there exist smooth local coordinates that conjugate  $X$ and $DX_p$ in a neighbourhood of $p$. In Lemma~\ref{linear} we prove that any $X \in \mathfrak{X}_\mu^1(M)$ having a singularity $p$ can be $C^1$-approximated by a vector field $Y \in \mathfrak{X}_\mu^{\infty}(M)$ such that $p$ is a linear hyperbolic singularity of $Y$. Since $Y$ is a divergence-free vector field it follows that $p$ is of the saddle-type. Now Proposition~\ref{V1} is a direct consequence of the following result.

\begin{proposition}(Proposition 4.1 of~\cite{V})
If $Y \in \mathfrak{X}^1(M)$ admits a linear hyperbolic singularity of saddle-type, then the linear Poincar\'e flow of $Y$ does not admit any dominated splitting  over $M\setminus Sing(Y)$.
\end{proposition}

\bigskip

Proposition~\ref{V2} is a consequence of the following result, which is an adaptation to the conservative setting of Corollary 2.22 of~\cite{BGV}, whose proof  is postponed to section~\ref{MG}.

\begin{proposition}\label{VF}
Let $X \in \mathfrak{X}_\mu^4(M)$ be a $C^1$-robustly transitive vector field and let $\mathcal{U}$ be a small $C^1$-neighbourhood of $X$. There exist $\ell, \varrho \in \mathbb{R}_0^+$ such that for any $Y \in \mathcal{U}$ and any periodic orbit $x$ of $Y^t$, of period $\pi(x) \geq \varrho$, the linear Poincar\'e flow of $Y$ admits an $\ell$-dominated splitting  over the $Y^t$-orbit of $x$.
\end{proposition}

Let us explain how Proposition~\ref{V2} is deduced from the previous result.

Fix  a robustly transitive vector field $X \in \mathfrak{X}_\mu^4(M)$ and let  $\mathcal{U}$ be a $C^1$-neighbourhood of $X$ as in the previous proposition and such that any $C^1$ vector field $Y \in \mathcal{U}$ is also robustly transitive. Consider $\ell$ and $\varrho$ given by  Proposition~\ref{VF}.

Let $x \in M$ be a point with dense $X^t$-orbit. Using Pugh and Robinson's volume-preserving closing lemma (\cite{PR}) we get a sequence of vector fields $X_n \in \mathfrak{X}_\mu^4(M)$, converging to $X$ in the $C^1$-topology, and, for each $n \in \mathbb{N}$, $X_n$ has a periodic orbit $\Gamma_n=\Gamma_n(t)$ of period $\pi_n$ such that $\lim_{n \rightarrow +\infty} \Gamma_n(0)=x$. In particular $\lim_{n \rightarrow +\infty} \pi_n=+\infty$. Therefore, for large $n$, we can apply Proposition~\ref{VF} to conclude that there is an $\ell$-dominated splitting for the linear Poincar\'e flow of $X_n$ over the orbit $\Gamma_n$; taking a subsequence if necessary, we can assume that the dimensions of the invariant bundles do not depend on $n$. Finally, as $\bigcap_N(\overline{\cup_N^{+\infty}\Gamma_n})=M$, by well known results on dominated splittings (see for example~\cite{BDV}) it follows that there exists an $\ell$-dominated splitting for the linear Poincar\'e flow of $X$ over $M\setminus Sing(X)$.

\end{section}

\begin{section}{Perturbation Lemmas}\label{pl}

In this section we state and prove three perturbation lemmas needed to obtain the main results of this article. In particular the Main Perturbation Lemma (Lemma~\ref{mpl}) is a kind of Franks' Lemma (\cite{F}) for conservative flows. As we mention before, for technical reasons, we required that the vector fields involved are of class $C^4$. One of the main tools to  obtain this result is the Arbieto and  Matheus Pasting Lemmas (\cite{AM}). We refer that their result and our Main Perturbation Lemma make use of a key result of Dacorogna and Moser (\cite{DM}).

We fix a vector field $X \in \mathfrak{X}_\mu^4(M)$, $\tau >0$, and a point $p \in M$ such that $X^t(p)\not=p$, for all $t \in [0, \tau]$. Define $\Gamma(p, \tau)=\{X^t(p);\,\, t \in [0, \tau]\}$. In the sequel up to a smooth conservative change of coordinates defined on a neighbourhood of $\Gamma(p,\tau)$ (see~\cite{Mo}) we can assume that we are working on  $\mathbb{R}^n$ and that $\frac{1}{\|X(p)\|}X(p)=\frac{\partial}{\partial x_1}=v$. Let $W\subset \mathbb{R}^n$ be the $(n-1)$-dimensional vectorial subspace orthogonal to the unitary vector $v$.

Given $r>0$ let $B_r(p)$ denote the ball of radius $r$, centered at $p$ and contained in $N_p=X(p)^{\perp}=W$. For  $r>0$ and $\delta >0$ define

$$\mathcal{T}=\mathcal{T}(p,\tau,r, \delta)=\bigcup_{t \in ]-\delta, \tau+\delta[} X^t(B_r(p)).$$

If $r>0$ and $\delta >0$ are small enough the set $\mathcal{T}$ is an open neighbourhood of $\Gamma(p,\tau)$; by definition this neighbourhood is foliated by orbits of the flow so we call it a \emph{flowbox}.

We fix a linear isometry $\iota_p \colon N_p \rightarrow W$  and choose a family $\{\iota_t\}_{t \in ]-\delta, \tau+\delta[}$, such that, for each $t \in ]-\delta, \tau+\delta[$, $\iota_t$ is a linear isometry from $N_{X^t(p)}$ onto $W$, $\iota_0=\iota_p$, and this family is $C^1$ on the parameter $t$. 

Given $q \in \mathcal{T}$, we write $q=\lambda_q v+w_q$, where $w_q \in W$ and $\lambda_q \in \mathbb{R}$. Define $\ell(t)=\int_0^t \|X(X^s(p))\|ds$; there exists $t_q \in ]-\delta, \tau + \delta[$  such that $\ell(t_q)=\lambda_q$. We note that $t_p=0$.

Let us now define the \emph{Poincar\'e flow} $\hat{X}^t$ associated to $X$ on $\mathcal{T}$. 

For $t$ such that $t_q+t \in ]-\delta, \tau + \delta[$ define

$$\hat{X}^t(q)= \ell(t_q+t)v+ \iota_{(t_q+t)}\circ P_X^{t}(X^s(p))\circ \iota_{t_q}^{-1}(w_q).$$
It is straightforward to see that $\hat{X}^0 \equiv Id$ and that 
$\hat{X}^{t+t^\prime}(q)=\hat{X}^t(\hat{X}^{t^\prime}(q))$, when defined. 

Let $\hat{X}$ be the  vector field associated to the flow $\hat{X}^t$. $\hat{X}$ is of class $C^2$ and it is divergence-free. To see this we first recall the Liouville formula
$$exp \left(\int_0^t div(\hat{X}(\hat{X}^s(q)))ds\right) =\det D\hat{X}^t(q).$$
Now a direct computation gives that the  matrix of $D\hat{X}^t(q)$ relatively to the decomposition $\mathbb{R}^n=W\oplus \langle v \rangle $ is
$$\begin{pmatrix} \frac{\|X(X^{t_q+t}(p))\|}{\|X(X^{t_q}(p))\|} & 0 \\ \ast & \iota_{t_q+t}\circ P_X^{t}(X^{t_q}(p))\circ \iota_{t_q}^{-1} \end{pmatrix}.$$
As $X^t$ is volume-preserving and the maps $\iota_s$ are linear isometries we get
$$\det D\hat{X}^t(q)=\frac{\|X(X^{t_q+t}(p))\|}{\|X(X^{t_q}(p))\|} \times \det P_X^{t}(X^{t_q}(p))=1,\,\,\forall t.$$
Thus, according to  Liouville's formula, it follows that  $div(\hat{X})=0$.

We also observe that $P_{\hat{X}}^t(q)=\iota_{t_q+t}\circ P_X^{t}(X^{t_q}(p))\circ \iota_{t_q}^{-1}$; in particular
$P_{\hat{X}}^t(0)=\iota_{t}\circ P_X^{t}(p) \circ \iota_0^{-1}$.

\begin{lemma}\label{hat}
Let $X \in \mathfrak{X}_\mu^4(M)$, $\tau >0$, and $p \in M$ such that $X^t(p)\not=p$, $\forall t \in [0, \tau]$.
There exists a $C^2$-conservative change of coordinates $\Phi$, defined on a neighbourhood of $\Gamma(p,\tau)$, such that $$\hat{X}=\Phi_{*}X \text{ and } \Phi(X^t(p))=\hat{X}^t(0),\,\,\forall t \in [0, \tau].$$

\end{lemma}
\begin{proof}
We recall that $\Phi_{*}X(y)=D\Phi_{\Phi^{-1}(y)}X(\Phi^{-1}(y))$.
By Lemma 2.1 of~\cite{BR} we know that
there exists a conservative $C^{2}$
diffeomorphism $\Psi$ defined on a flowbox containing $\Gamma(p,\tau)$ such that $T=\Psi_{*}X$, where $T= \frac{\partial}{\partial x_1}$. Exactly in the same way there exists a conservative $C^{2}$
diffeomorphism $\hat{\Psi}$ defined on a flowbox containing $\Gamma(0,\tau)$ such that $T=\hat{\Psi}_{*}\hat{X}$. 

Up to translations defined on the hyperplane $\langle \frac{\partial}{\partial x_1} \rangle^{\perp}$ and shrinking the neigbourhoods of the definition of these maps, we can assume that $\Psi(p)=0$ and that $\hat{\Psi}(0)=0$. Finally, we define $\Phi=\hat{\Psi}^{-1}\circ \Psi$.
\end{proof}

\bigskip

Now let $V, V^{\prime} \subset N_p$, $dim(V)=j$, $2 \leq j \leq n-1$, and  $N_p=V \oplus V^{\prime}$. A \emph{one-parameter linear family} $\{A_t\}_{t\in \mathbb{R}}$ associated to $\Gamma(p, \tau)$ and $V$  is defined as follows:
\begin{itemize}
\item $A_t\colon N_p \rightarrow N_p$ is a linear map, for all $t\in \mathbb{R}$,
\item $A_t=Id$, for all $t\leq 0$, and $A_t=A_{\tau}$, for all $t\geq \tau$,
\item $A_t|_V \in SL(j, \mathbb{R})$, and $A_t|_{V^{\prime}}\equiv Id$, $\forall t \in [0, \tau]$, in particular we have  $\det(A_t)=1$, for all $t\in \mathbb{R}$, and
\item the family $A_t$ is $C^\infty$ on the parameter $t$.
\end{itemize}

\begin{lemma}(Main perturbation lemma)\label{mpl}

Given $\epsilon>0$ and a vector field $X \in \mathfrak{X}_\mu^4(M)$ there exists $\xi_0=\xi_0(\epsilon,X)$ such that $\forall \tau \in [1,2]$, for any periodic point $p$ of period greater than $2$, for any sufficient small flowbox $\mathcal{T}$ of $\Gamma(p, \tau)$ and for any one-parameter linear family $\{A_t\}_{t \in [0, \tau]}$ such that $\|  A_t^{\prime} A_t^{-1}\|<\xi_0$, $\forall t \in [0, \tau]$, there exists $Y \in \mathfrak{X}_\mu^1(M)$ satisfying the following properties
\begin{enumerate}
\item $Y$ is $\epsilon$-$C^1$-close to $X$;
\item $Y^t(p)=X^t(p)$, for all $t \in \mathbb{R}$;
\item $P_Y^\tau(p)=P_X^\tau(p) \circ A_{\tau}$, and
\item $Y|_{\mathcal{T}^c}\equiv X|_{\mathcal{T}^c}$.
\end{enumerate}
\end{lemma}
\begin{proof}

Using Lemma~\ref{hat} we get a $C^2$ change of coordinates $\Phi$, defined in a flowbox $\mathcal{T}=\mathcal{T}(p, \tau, r, \delta)$ and such that  $\Phi_{*}X=\hat{X}$ and $\Phi(X^t(p))=\hat{X}^t(0)$, $\forall t \in [0, \tau]$, where $\hat{X}^t$ is the Poincar\'e flow associated to $X$ on $\mathcal{T}$ defined above. 

To obtain the vector field $Y$ we first construct a $C^2$ divergence-free vector field $\hat{Y}$ defined on $\Phi(\mathcal{T})$ and such that
\begin{description}
\item (a) $\hat{Y}$ is $\hat{\epsilon}$-$C^1$-close to $\hat{X}$;
\item (b) $\hat{Y}^t(0)=\hat{X}^t(0)$, when defined;
\item (c) $P_{\hat{Y}}^t(0)=P_{\hat{X}}^t(0) \circ B_{t}$, where $B_{t}=\iota_p \circ A_t$, $t\in [0, \tau]$, and
\item (d) $\hat{Y}|_{\hat{\mathcal{T}}^c}\equiv \hat{X}|_{\hat{\mathcal{T}}^c}$, where ${\hat{\mathcal{T}_2}}=\Phi(\mathcal{T}(p, \tau, r_2, \delta_2))$, for some $0<r_2<r$ and $0<\delta_2<\delta$ to be fixed.
\end{description}

The positive real number $\hat{\epsilon}$ depends only on $\Phi$ and $\mathcal{T}$ and assures that if $Z$ is $\hat{\epsilon}$-$C^1$-close to $\hat{X}$ on $\Phi(\mathcal{T})$ then ${\Phi^{-1}}_{\ast}(Z)$ is $\epsilon$-$C^1$-close to $X$ on $\mathcal{T}$.

Once we get $\hat{Y}$ we define $\tilde{Y}={\Phi^{-1}}_{\ast}(\hat{Y})$ and, as $\tilde{Y}\equiv X$ on $\mathcal{T} \setminus \mathcal{T}(p, \tau, \frac{r}{2}, \frac {\delta} {2})$, we consider $Y=\tilde{Y}$ on $\mathcal{T}$ and $Y=X$ on $\mathcal{T}^c$. We observe that we can only guaranty that $Y$ is of class $C^1$.

From this construction it follows immediately that items \emph{1.}, \emph{2.} and \emph{4.} of the lemma are a direct consequence of conditions \emph{(a)}, \emph{(b)} and \emph{(d)} on $\hat{Y}$, respectively. 

To get item \emph{3.} we observe that our construction  of $\hat{Y}$ will imply that $P_{\hat{Y}}^t(p)=\iota_t \circ P_{\tilde{Y}}^t(p) \circ \iota_p^{-1}$ where, recall, $\tilde{Y}={\Phi^{-1}}_{\ast}(\hat{Y})$. Therefore, as $P_{\hat{X}}^t(p)=\iota_t \circ P_{X}^t(p) \circ \iota_p^{-1}$ and $B_{t}=\iota_p \circ A_t$, from condition \emph{(c)} we obtain \emph{3.}.

Let us now explain how to  construct the vector field $\hat{Y}$ defined on $\Phi(\mathcal{T})$.

The linear variational equation associated to the linear Poincar\'e flow of $\hat{X}$ is
$$[P_{\hat{X}}^t(0)]^\prime=(\Pi \circ D\hat{X}_{\hat{X}^t(0)})(P_{\hat{X}}^t(0)),$$
where $^\prime$ denotes the time derivative, $D$ is the spacial derivative and $\Pi$ is the orthogonal projection onto $N_{\hat{X}^t(0)}$. To get $\hat{Y}$ we begin by considering an analogous linear variational equation associated to $P_{\hat{Y}}^t(0)$ in order to obtain $D\hat{Y}$ along the orbit of the point $0$ and then define in a linear way the flow $\hat{Y}^t$.

Since we require that $P_{\hat{Y}}^t(0)=P_{\hat{X}}^t(0) \circ B_{t}$ we have that

\begin{eqnarray*}
[P_{\hat{Y}}^t(0)]^\prime&=&[P_{\hat{X}}^t(0) \circ B_{t}]^\prime=[P_{\hat{X}}^t(0)]^\prime \circ B_{t} +  P_{\hat{X}}^t(0) \circ {B_{t}}^{\prime}\\
&=& (\Pi \circ D\hat{X}_{\hat{X}^t(0)})(P_{\hat{X}}^t(0)) \circ B_t +  P_{\hat{X}}^t(0) \circ {B_{t}}^{\prime}\\
&=& \left[\Pi \circ D\hat{X}_{\hat{X}^t(0)}+ (P_{\hat{X}}^t(0) \circ {B_{t}}^{\prime}) \circ (B_t^{-1} \circ P_{\hat{X}}^{-t}(\hat{X}^t(0)))\right]\circ P_{\hat{Y}}^t(0).
\end{eqnarray*}
This equation allows us to define, along the $\hat{X}^t$ orbit of $0$, the infinitesimal generator $D\hat{X}+\mathcal{H}$ of the desired perturbation, where $\mathcal{H}_{\hat{X}^t(0)}(v)=\vec{0}$ and
\begin{equation}\label{C}
\Pi \circ \mathcal{H}_{\hat{X}^t(0)}=(P_{\hat{X}}^t(0) \circ {B_{t}}^{\prime}) \circ (B_t^{-1} \circ P_{\hat{X}}^{-t}(\hat{X}^t(0)))=C_t.
\end{equation}

With this definition the previous differential equation can be written as:
\begin{equation}\label{D}
u^{\prime}(t)=\Pi \circ (D\hat{X}+ \mathcal{H})_{\hat{X}^t(0)}(u(t)).
\end{equation}
Let $\mathcal{P}(\lambda v, w)=(0, C_t(w))$, where $t$ is given by $\int_0^t\|\hat{X}(\hat{X}^s(0))\|ds= \lambda$; observe that 
\begin{equation}\label{W}
D_w\mathcal{P}(\hat{X}^t(0))=\Pi \circ \mathcal{H}_{\hat{X}^t(0)},
\end{equation}
and that 
\begin{equation}\label{U}
D_w\mathcal{P}(\hat{X}^t(0))(0,u)= \Pi \circ D\mathcal{P}_{\hat{X}^t(0)}(0,u).
\end{equation}

Now we define the $C^2$ vector field $\hat{Y}(q)=(\hat{X} + \mathcal{P})(q)$,  for  $q \in \Phi(\mathcal{T}(p,\tau, r_1, \delta_1))$, where $0<r_1<r$ and $0<\delta_1<\delta$ will be fixed. In order to extend this vector field let us first prove that it is divergence-free.

As $\hat{Y}=\hat{X} + \mathcal{P}$ and $\hat{X}$ is divergence-free, using (\ref{W}), (\ref{C}) and the definition of the maps $B_t$, it follows that
$$div(D\hat{Y})=div(D\mathcal{P})=tr(C_t)=tr(B_t^{\prime} \circ B_t^{-1})=tr(A_t^{\prime} \circ A_t^{-1}).$$
Now, as $\det(A_t)=1, \forall t \in \mathbb{R}$, the result follows observing that
$$0=(\det(A_t))^{\prime}=tr(A_t^{\prime}A_t^{-1})\det(A_t)=tr(A_t^{\prime}A_t^{-1}).$$
Now, to extend  $\hat{Y}$ to a conservative vector field, we apply the Arbieto and Matheus \emph{$C^{1+\alpha}$ -Pasting Lemma} (Theorem 3.1,~\cite{AM}) which guaranties that there are $0<r_1<r_2<r$ and $0<\delta_1<\delta_2<\delta$ such that $\hat{Y}$ has a divergence-free $C^2$ extension to $\Phi(\mathcal{T}(p,\tau, r, \delta))$, that we also denote by $\hat{Y}$, with $\hat{Y}(q)=\hat{X}(q)$, for all $q \in \Phi(\mathcal{T}(p,\tau, r, \delta)) \setminus \Phi(\mathcal{T}(p,\tau, r_2, \delta_2))$.

Let us now prove that this vector field $\hat{Y}$ satisfies properties \emph{(a)}-\emph{(d)}.

Condition \emph{(d)} is a direct consequence of the way we made the extension of $\hat{Y}$. To get \emph{(b)} just observe that 
\begin{eqnarray*}
\hat{Y}(\hat{X}^t(0))&=&\hat{X}(\hat{X}^t(0))+\mathcal{P}\left(\int_0^t \|\hat{X}(  \hat{X}^s(0))\|ds,\vec{0}\right)\\
&=&\hat{X}(\hat{X}^t(0))+(0,C_t(\vec{0}))=\hat{X}(\hat{X}^t(0)).
\end{eqnarray*}

To get \emph{(c)} let us first remark that the linear Poincar\'e flow of $\hat{Y}$ at $0$, $P_{\hat{Y}}^t(0)$, is the solution of the differential equation $u^{\prime}(t)=\Pi \circ D\hat{Y}_{\hat{Y}^t(0)}(u(t))$. By equations~(\ref{U}) and~(\ref{W}) we have that 
$$ \Pi \circ D\hat{Y}_{\hat{Y}^t(0)}(u(t))=\Pi \circ D(\hat{X}+\mathcal{P})_{\hat{X}^t(0)}=\Pi \circ D\hat{X}_{\hat{X}^t(0)}+ \Pi \circ \mathcal{H}_{\hat{X}^t(0)}.$$
Hence, by~(\ref{C}) and~(\ref{D}), we get that $P_{\hat{Y}}^t(0)=P_{\hat{X}}^t(0) \circ B_{t}$.

To prove condition \emph{(a)} we begin by observing that the Pasting Lemma guaranties that given $\hat{\epsilon}$ there exists $\sigma>0$ such that if the vector field $\hat{Y}$ we constructed is $\sigma$-$C^1$ close to $\hat{X}$ on $\Phi(\mathcal{T}(p,\tau, r_1, \delta_1))$ then its extension to $\Phi(\mathcal{T}(p,\tau, r, \delta))$ is $\hat{\epsilon}$-$C^1$ close to $\hat{X}$, and to take $r_1$ and $\delta_1$ smaller does not change this $C^1$ closeness.
So let us prove that $\hat{Y}$ is $\sigma$-$C^1$ close to $\hat{X}$ on $\mathcal{T}_1=\Phi(\mathcal{T}(p,\tau, r_1, \delta_1))$ for sufficiently small $r_1$, $\delta_1$ and $\xi_0$.

Recalling that $\hat{Y}=\hat{X}+ \mathcal{P}$, that $\mathcal{P}(\hat{X}^t(0), \vec{0})=0$ and that $\mathcal{P}$ is continuous, to choose $r_1$ and $\delta_1$ small is enough to assure $\|\hat{Y}-\hat{X}\|_0< \sigma$ on $\mathcal{T}_1$. 

We observe that the matrix of $D\mathcal{P}(\lambda,w)$ depends only on  map $C_t$ and on real numbers $$\left|\frac{\partial [C_t(w)]^i}{\partial t} \frac{\partial t}{\partial  \lambda}w_i \right|, \,\,\,i \in \{1,...,n-1\},$$ where $w=(w_1,...,w_{n-1})$ and $\|w\|<r_1$. By the definition of $C_t$~(see~(\ref{C})), up to constants that depend only on $X$, its norm is given by $$\|B^{\prime}_t\circ B_t^{-1}\|=\|\iota_p \circ A^{\prime}_t\circ A_t^{-1} \circ \iota_p^{-1}\|= \|A^{\prime}_t\circ A_t^{-1}\|,$$  
because the map $\iota_p$ is an isometry. Hence, by the hypothesis, it is enough to take $\xi_0$ sufficiently small
to get that $\|D\mathcal{P}\| < \sigma$. This ends the proof of the lemma. \end{proof}

\begin{lemma}\label{linear}
Let $p$ be a singularity of $X \in \mathfrak{X}_\mu^1(M)$. For any $\epsilon >0$ there exists  $Y \in  \mathfrak{X}_\mu^{\infty}(M)$, such that $Y$ is $\epsilon$-$C^1$-close to $X$ and  $p$ is a linear hyperbolic singularity of $Y$.
\end{lemma}
\begin{proof}
Let $(U,\phi)$ be a conservative chart given by Moser's Theorem (\cite{Mo}) such that $p \in U$ and $\phi(p)=0$. Let $A=DX_p$ and, for arbitrarily small $\tilde{\delta}>0$, choose a linear and hyperbolic isomorphism $H=H_{\delta}$ such that $\|A-H\|<\delta$. We fix small $r>0$ such that $B(0;r)\subset \phi(U)$ and consider the pull-back of $H$,  $Z=Z_{\tilde{\delta},r}={(\phi^{-1})}_\ast H$, defined on $\phi^{-1}(B(0;r))$.

For any small $\delta>0$ there are $\tilde{\delta}>0$ and $r>0$ such that $X$ and $Z$ are $\delta$-$C^1$-close. Therefore a straightforward application of the $C^1$-pasting lemma (Theorem 3.2, \cite{AM}) to $X$ and $\epsilon$ (which give a $\delta$) guarantees that there exists a vector field $Y \in  \mathfrak{X}_\mu^{\infty}(M)$ such that $Y$ is $\epsilon$-$C^1$-close to $X$ and $Y|_W=Z_W$, where $W=\phi^{-1}(B(0; \frac{r}{2}))$. From the construction it follows that $p$ is a linear hyperbolic singularity of $Y$.
\end{proof}

\begin{lemma}\label{linear2}
Let $X \in \mathfrak{X}_\mu^1(M)$ and assume that $p \in M$ is a periodic elliptic point of period $\pi(p)> 1$. Then,
for any $\epsilon >0$  there exist $Z \in  \mathfrak{X}_\mu^{\infty}(M)$ and a tubular neighbourhood $\mathcal{U}$ of the $X^t$-orbit of $p$, such that $Z$ is $\epsilon$-$C^1$-close to $X$ and  $\mathcal{U}$ is $Z^t$-invariant.
\end{lemma}
\begin{proof}

Let $p \in M$ is a periodic elliptic point of period $\pi(p)> 1$ of $X$ and denote its orbit by $\gamma=\Gamma(p, \pi(p))$. As $p$ is elliptic all the eigenvalues of the linear Poincar\'e map $P_X^{\pi(p)}(p):N_p \rightarrow N_p$ have modulus one. 

Let us first assume that  the map $P_X^{\pi(p)}(p)$ admits a basis formed by eigenvectors. We consider the inner product defined on $N_p$ and associated to this basis, that is the one that orthonormalizes the fixed basis.
 For $r>0$ let $D(0;r) \subset N_p$ denote the $(n-1)$-disk  centered at $0 \in N_p$ and of radius $r$ for the distance associated to this inner product. Note that $D(0;r)$ is $P_X^{\pi(p)}(p)$-invariant.

Let $Y$ be the divergence-free linear vector field associated to the flow obtained by suspending $P_X^{\pi(p)}(p)$ along $\gamma$; this vector field is defined in a tubular neighbourhood  of $\gamma$, $\mathcal{U}(r)$, which is homeomorphic to $\gamma \times D(0;r)$.

Given $\delta>0$ we can choose a small $r$ such that $X$ and $Y$ are $\delta$-$C^1$-close on $\mathcal{U}(r)$. Now, for fixed $\epsilon$ and an appropriate $\delta$ we apply the $C^1$-pasting lemma (Theorem 3.2, \cite{AM}) to get a vector field $Z \in  \mathfrak{X}_\mu^{\infty}(M)$ such that $X$ and $Z$ are  $\epsilon$-$C^1$-close on $M$ and $Z|_{\mathcal{U}}=Y|_{\mathcal{U}}$, where $\mathcal{U}=\mathcal{U}(\frac{r}{2})$. It follows from this construction that $\mathcal{U}$ is $Z^t$-invariant.

Assume now that $P_X^{\pi(p)}(p)$ does not admit a basis formed by eigenvectors, that is there exists at least one eigenvalue whose multiplicity is bigger than the dimension of the associated eigenspace. Let us first explain how to deal with the simplest case, that is when $dim(M)=3$, $P_X^{\pi(p)}(p)$ has only one eigenvalue, say equal to $1$, and the associated eigenspace is one-dimensional. In this case we will perturb $X$ in order to get complex eigenvalues. 

For that and exactly as before we begin by  $\frac{\epsilon}{2}$-$C^1$-approximate $X$ by a $C^{\infty}$ vector field $Y$ linear in a neighbourhood of $\gamma$ and such that $P_X^{\pi(p)}(p)=P_Y^{\pi(p)}(p)$.
Now, there exists a basis of $N_p$ such that relatively to this basis $P_Y^{\pi(p)}(p)$ has matrix 
$$
\begin{pmatrix}
1&0\\1&1
\end{pmatrix}.$$
Consider the one-parameter linear family  
$$A_t=
\begin{pmatrix}
1&-\delta \alpha(t) \\0&1
\end{pmatrix},$$
where $\alpha(t)$ is a $C^{\infty}$-bump function, $\alpha(t)=1$, for all $t \geq 1$, $\alpha(t)=0$, for all $t \leq 0$, and $\delta>0$ is arbitrarily small. 

As $\|A_t^{\prime} \circ A_t^{-1}\| \leq \delta |\alpha^{\prime}(t)|$, choosing $\delta$ small enough we can apply Lemma~\ref{mpl} to the arc $\Gamma(p,1)$ in order to obtain a new  vector field $Z \in \mathfrak{X}_\mu^{\infty}(M)$, $\frac{\epsilon}{2}$-$C^1$-close to $Y$,  such that $p$ is a periodic orbit of $Z$ of period $\pi(p)$ and
$P_Z^1(p)=P_Y^1(p)\circ A_1=P_X^1(p)\circ A_1$. As 
$$P_Z^{\pi(p)}(p)= P_Z^{\pi(p)-1}(Z^1(p)) \circ  P_Z^1(p)=P_X^{\pi(p)-1}(X^1(p)) \circ  P_X^1(p) \circ A_1,$$
it follows that the matrix of  $P_Z^{\pi(p)}(p)$ with respect to the basis we fixed above is
$$
\begin{pmatrix}
1&-\delta\\1&1-\delta
\end{pmatrix}.$$
Therefore $p$ is an elliptic point of $Z$ and $P_Z^{\pi(p)}(p)$ has two complex eigenvalues. Moreover, it is clear that $Z$ and $X$ are $\epsilon$-$C^1$-close.

If the eigenvalue is equal to $-1$ we proceed in the same way considering the matrix
$$A_t=
\begin{pmatrix}
1&\delta \alpha(t) \\0&1
\end{pmatrix}.$$

Finally, to deal with the general case (several eigenvalues, real or complex, whose multiplicity is greater than the dimension of the corresponding eigenspace) we just have to apply the previous argument a finite number of times in order to get a new vector field $Z$ such that $p$ is a periodic elliptic orbit, of period $\pi(p)$, and such that $P_Z^{\pi(p)}(p)$ admits a basis of eigenvectors.
\end{proof}

\end{section}

\begin{section}{Proof of Proposition~\ref{VF}}\label{MG}

\begin{subsection}{Conservative Linear Differential Systems}
We begin this section by recalling some definitions  introduced in~\cite{BGV}.

Let $X \in \mathfrak{X}_\mu^1(M)$ and consider a set $\Sigma \subset M$ which is a countable union of periodic orbits of $X^t$. A \emph{Linear Differential System} (LDS) is a four-tuple $\mathcal{A}=(\Sigma, X^t, N_{\Sigma}, A^{\star})$, where $N_{\Sigma}$ is the restriction to $\Sigma$ of the normal bundle of $X$ over $M\setminus Sing(X)$ and $A^{\star} \colon  \Sigma \rightarrow GL(n-1,\mathbb{R})$ is a continuous map. In fact, for $x \in \Sigma$, $A^{\star}_x$ is a linear map of $N_x$ and we identify this space with $\mathbb{R}^{n-1}$. The natural LDS associated to the dynamics of the vector field is obtained by taking $A^{\star}_x=\Pi \circ DX_x$. 

Given an LDS $\mathcal{A}=(\Sigma, X^t, N_{\Sigma}, A^{\star})$ the \emph{linear variational equation} associated to it is
\begin{equation}\label{lve}
\dot{u}(t,x)=A^{\star}(X^t(x)) \cdot u(t,x).
\end{equation}

The solution of (\ref{lve}) with initial condition $u(0,x)=Id$ is, for each $t$ and $x$, a linear map $\Phi_{A^{\star}}^t(x) \colon N_x \rightarrow N_{X^t(x)}$. The map $A^{\star}$ is called the \emph{infinitesimal generator} of $\Phi_ {A^{\star}}$; it is easy to see that  $\Phi_{A^{\star}}^t(x)=P_X^t(x)$ when the infinitesimal generator is $\Pi \circ DX$.

The LDS $\mathcal{A}=(\Sigma, X^t, N_{\Sigma}, A^{\star})$ is \emph{bounded} if there exists $K>0$ such that $\|A^{\star}_x\| \leq K$, for all $x \in \Sigma$. The LDS $\mathcal{A}$ is said to be a \emph{large period system} if the number of orbits of $\Sigma$ with period less or equal to $\tau$ is finite, for any $\tau >0$. 

We say that the LDS $\mathcal{A}$ is \emph{conservative} if 
$$ | \det\Phi_{A^{\star}}^t(x) | = \frac{\|X(x)\|}{\|X(X^t(x))\|},\,\, \forall x \in \Sigma.$$
We observe that from Liouville's formula it follows that 
\begin{equation}\label{LF}
\det\Phi_{A^{\star}}^t(x) = \exp\left(\int_0^t tr(A^{\star}(X^s(x)))ds\right).
\end{equation}

A  LDS $\mathcal{B}=(\Sigma, X^t, N_{\Sigma}, B^{\star})$  is a \emph{conservative perturbation} of a bounded $\mathcal{A}$ if, for every $\epsilon >0$, $\|A^{\star}_x - B^{\star}_x\| < \epsilon$,  up to points $x$ belonging to a finite number of orbits, and $\mathcal{B}$ is conservative. In view of (\ref{LF}) it follows that $\mathcal{B}$ is conservative if and only if $tr(B^{\star})=tr(A^{\star})$.

A direct application of the Gronwall inequality gives that $$\|\Phi_{A^{\star}}^t(x)-\Phi_{B^{\star}}^t(x)\| \leq \exp(K|t|)\|A^{\star}_x - B^{\star}_x\|.$$
In particular $\Phi_{B^{\star}}^1$ is a perturbation of $\Phi_{A^{\star}}^1$ in the sense introduced in~\cite{BGV} for the discrete case.

A bounded LDS $\mathcal{A}$ is \emph{strictly without dominated decomposition} if the only invariant subsets of $\Sigma$ that admit a dominated splitting for $\Phi_{A^{\star}}^t$ are finite sets.

Let us now present a key result about linear differential systems which is the conservative version of Theorem 2.2 of~\cite{BGV}.

\begin{theorem}\label{2.2}
Let $\mathcal{A}$ be a conservative, large period and bounded LDS. If $\mathcal{A}$ is strictly without dominated decomposition then there exist a conservative perturbation $\mathcal{B}$ of $\mathcal{A}$ and an infinite set $\Sigma^{\prime} \subset \Sigma$ which is  $X^t$-invariant such that for every $x \in \Sigma^{\prime}$ the linear map $\Phi_{B^{\star}}^{\pi(x)}(x)$ as all eigenvalues real and with the same modulus (thus equal to $1$ or to $-1$).
\end{theorem}

The perturbations used in the proof of Theorem 2.2 of~\cite{BGV} are rotations and directional homotheties (diagonal linear maps for a fixed basis). They are made in the linear cocycle setting and (discrete/continuous-time) Franks' Lemma allows to realize them as perturbations of a fixed diffeomorphism or vector field. Once we  have a dictionary to pass from linear cocycles (discrete case) to conservative linear differential systems (\emph{conservative} continuous-time case) and we obtained the Main Perturbation Lemma (Lemma~\ref{mpl}) which allows to realize these kind of conservative perturbations of linear differential systems as conservative perturbations of vector fields, the proof given by Bonatti, Gourmelon and Vivier can be carried on to our setting without additional obstructions. 
Therefore, to illustrate how this can be done,  we show how to perturb along a periodic orbit of a conservative vector field in order to get real eigenvalues for the linear Poincar\'e map in the period. This is obtained by first making a conservative perturbation of the  LDS associated to the orbit and then, using the Main Perturbation Lemma, realize it as a conservative perturbation of the vector field. 

\begin{lemma}
Let $X \in \mathfrak{X}_\mu^4(M)$ and fix small $\epsilon_0 >0$. There exists $\pi_0$ such that for any periodic orbit $x$ with period $\pi(x)>\pi_0$ there is $Y \in \mathfrak{X}_\mu^1(M)$ satisfying
\begin{itemize}
\item $Y$ is $\epsilon_0$-$C^1$-close to $X$;
\item $Y^t(x)=X^t(x)$, $\forall t \in \mathbb{R}$;
\item all the eigenvalues of $P_Y^{\pi(x)}$ are real, and 
\item $Y$ is equal to $X$ outside a small neighbourhood of the orbit of $x$.  
\end{itemize}
\end{lemma}
\begin{proof}

Let us fix a small $\delta >0$.

Let $R_{\theta}$ denote the rotation of angle $\theta$ in the plane. Lemma 6.6 of~\cite{BC} assures that there exists $N=N(\epsilon) \in \mathbb{N}$ satisfying the following:  for any $k>N$ and for any $C_1$, $C_2$,...,$C_k \in SL(2,\mathbb{R})$ there are rotations $R_{\theta_1}$,   $R_{\theta_2}$,...,$R_{\theta_k}$, with $|\theta_j|< \delta$ for all $j \in \{1,2,...,k\}$, such that the linear map 
$$ C_k \circ R_{\theta_k} \circ C_{k-1} \circ R_{\theta_{k-1}} \circ ... \circ C_1 \circ R_{\theta_1}$$ has real eigenvalues.

Let us fix a periodic orbit $\gamma$ and $x \in \gamma$ with $\pi(x) \geq N$. We assume that $P_X^{\pi(x)}(x)$ has a complex eigenvalue  associated to a two dimensional invariant subspace $V_x \subset N_x$. Assuming that $\pi(x)=k \in \mathbb{N}$, we consider the linear maps $C_j \colon V_{X^{j-1}(x)} \rightarrow V_{X^{j}(x)}$ defined by 
$$C_j=\frac {1}{\det(P_X^1(X^{j-1}(x))|_{V_{j-1}})}P_X^1(X^{j-1}(x))|_{V_{j-1}},$$ 
where   $V_j=P_X^j(V_x)$ and $j \in \{1,2,...k\}$. If $\pi(x) \notin \mathbb{N}$ we take $k=[\pi(x)]$, consider $C_1,...,C_{k-1}$ as before and define $C_k \colon V_{X^{k-1}(x)} \rightarrow V_{X^{\pi(x)}(x)}=V_{x}$ by 
$$C_k=\frac {1}{\det(P_X^{1+\pi(x)-k}(X^{k-1}(x))|_{V_{k-1}})}P_X^{1+\pi(x)-k}(X^{k-1}(x))|_{V_{k-1}}.$$ 
In what follows, without loss of generality, we assume that $\pi(x)=k \in \mathbb{N}$. 

We observe that each $C_j$ can be identified with a linear map of $SL(2, \mathbb{R})$ and that $P_X^{\pi(x)}(x)=C_k \circ C_{k-1}\circ...\circ C_1$. Therefore, Lemma 6.6 of~\cite{BC} gives a family of  rotations $R_{\theta_1}$,   $R_{\theta_2}$,...,$R_{\theta_k}$ with the properties described above. 

Now we want to apply  Lemma~\ref{mpl} to each arc $\Gamma(X^{j-1}(x),1)$ and to the maps $C_j$ and $R_{\theta_j}$. For that we consider $V=V_{j-1}$ and choose an appropriate $ V_{j-1}^{\prime}$ using the Jordan canonical form so that the perturbation we will construct do not change the other eigenvalues of $P_X^{\pi(x)}$. Then, for each $j \in \{1,2,...,k\}$, we define the one-parameter linear family $\{R_{j,t}\}_{t\in \mathbb{R}}$ associated to $\Gamma(X^{j-1}(x), 1)$ and $V_j$  by
\begin{itemize}
\item $R_{j,t} \colon N_{X^{j-1}(x)} \rightarrow N_{X^{j-1}(x)}$ is a linear map, for all $t\in \mathbb{R}$,
\item $R_{j,t}=Id$, for all $t\leq 0$, and $R_{j,t}=R_{j,1}$, for all $t\geq 1$,
\item $R_{j,t}|_{V_j^{\prime}}\equiv Id$, $\forall t \in [0,1]$, and
\item $R_{j,t}|_{V_j}=R_{\alpha_j(t)\theta_j}$, where $\alpha_j$ is a $C^{\infty}$ bump function with $\alpha_j(t)=0$ for $t \leq 0$, $\alpha_j(t)=1$ for $t \geq 1$, and $0 \leq \alpha_j^{\prime}(t) \leq 2$, for all $ t \in \mathbb{R}$.
\end{itemize}
A direct computation gives that $\|R_{j,t}^\prime \circ R_{j,t}^{-1}\|=\alpha^{\prime}(t) \theta_j \leq 2 \epsilon$. Therefore we fix $\delta \leq \frac{1}{2}\xi_0(\epsilon_0 /n, X)$, where $n$ is the dimension of $M$ and  $\xi_0(\epsilon_0 /n, X)$ is given by Lemma~\ref{mpl}; thus, applying this lemma we get divergence-free vector fields $Y_1$,...,$Y_{k}$, each one $\epsilon_0/n$-$C^1$-close to $X$  and such that 
$P_{Y_j}^1(X^{j-1}(x))= P_X^1(X^{j-1}(x)) \circ R_{\theta_j}$, for $j \in \{1,2,...,k\}$. It follows from this construction that these vector fields \emph{glue} together defining  a $C^1$ vector field $Y \in  \mathfrak{X}_\mu^1(M)$, $\epsilon_0/n$-$C^1$-close to $X$, and such that 
\begin{eqnarray*}
P_Y^{\pi(x)}(x)|_{V_x}&=&(P_X^1(X^{k-1}(x)) \circ R_{\theta_k}) \circ  ... \circ (P_X^1(x) \circ R_{\theta_1})\\
&=&(C_{k} \circ R_{\theta_k}) \circ (C_{k-1} \circ R_{\theta_{k-1}}) \circ ... \circ (C_1 \circ R_{\theta_1}), 
\end{eqnarray*}
therefore this linear map has real eigenvalues.
Finally, we apply these arguments at most $[\frac{n}{2}]$ times (corresponding to the maximal number of two-dimensional eigenspaces associated to complex eigenvalues) to get the vector field $Y$ satisfying the conditions of the lemma.
\end{proof}
\begin{remark} In the previous lemma we can assure that the eigenvalues of $P_Y^{\pi(x)}$ are all real and with different modulus. In fact this can be achieved by adding small directional homotheties in the two-dimensional vector spaces $V_x$.

\end{remark}

\end{subsection}

\begin{subsection}{Proof of Proposition~\ref{VF}}

We first note that it is not difficult to see that once we obtain the conclusions of the proposition for a robustly transitive vector field $X$ then they also hold for $Y$ in a small neighbourhood $\mathcal{U}$ of $X$ with the same $\ell$ and $\varrho$. 

So let us fix a robustly transitive $X \in \mathfrak{X}_\mu^4(M)$ and a positive number $\epsilon$ such that if $Y \in \mathfrak{X}_\mu^1(M)$ and is $\epsilon$-$C^1$-close to $X$ then it is robustly transitive. We assume that the proposition does not hold for $X$. Therefore for each $\ell \in \mathbb{N}$ there exists  a periodic point $x_{\ell}$, with period $\pi_{\ell}=\pi(x_{\ell})\geq \ell $ and such that the orbit $\Gamma(x_{\ell}, \pi_{\ell})$ does not admit an $\ell$-dominated splitting.

Define $$\Sigma=\bigcup_{\ell \in \mathbb{N}}  \Gamma(x_{\ell}, \pi_{\ell}),$$ and consider the linear differential system $\mathcal{A}=(\Sigma, X^t, N_{\Sigma}, \Pi \circ DX)$, that is $A^{\star}=\Pi \circ DX$ . As $M$ is compact and $X$ is $C^1$ $\mathcal{A}$ is bounded and, by construction, it  is  conservative and a large period system. 

Now we show that $\mathcal{A}$ is strictly without dominated decomposition. In fact let us assume that there exists an $X^t$-invariant and not finite  set $\Sigma_0 \subset \Sigma$ such that $P_X^t$ admits an $L_0$ dominated splitting over $\Sigma_0$, say $N_{\Sigma_0}=E \oplus F$. It follows that there is $L_1>L_0$ such that $E \oplus F$ is an $L$-dominated splitting for the linear Poincar\'e map, for any $L \geq L_1$. As $\Sigma_0$ is not finite there exists $\ell_0>L_1$ such that $\Sigma_1=\Gamma(x_{\ell_0}, \pi_{\ell_0}) \subset \Sigma_0$. The set $\Sigma_1$ is an $X^t$-invariant subset of $\Sigma_0$ therefore, by choice of $L_1$, it admits an $L$- dominated splitting for any $L \geq L_1$; by other side as $\pi_{\ell_0} \geq \ell_0 >L_1$ it follows that $\Sigma_1$ does not admit an $\ell_0$-dominated splitting, which is a contradiction.

Now we can apply Theorem~\ref{2.2} in order to get a conservative perturbation $\mathcal{B}$ of $\mathcal{A}$ and an infinite set $\Sigma^{\prime} \subset \Sigma$ which is  $X^t$-invariant such that for every $x \in \Sigma^{\prime}$ the linear map $\Phi_{B^{\star}}^{\pi(x)}(x)$ as all eigenvalues real and with the same modulus (thus equal to $1$ or to $-1$). As $\mathcal{B}$ is a perturbation of $\mathcal{A}$, for any small $\eta_0>0$  there exits $\overline{x} \in \Sigma^{\prime}$ such that $\|B^{\star}_{X_t(\overline{x})}-A^{\star}_{X_t(\overline{x})}\|<\eta_0$, for all $t \in \mathbb{R}$. We observe that the period of $\overline{x}$, $\pi(\overline{x})$ tends to infinity as  $\eta_0$ goes to zero.

Now we construct a new vector field $Y \in \mathfrak{X}_\mu^1(M)$ such $Y$ is $\epsilon$-$C^1$-close to $X$, $Y^t(\overline{x})=X^t(\overline{x})$, for all $t \in \mathbb{R}$ and $P_Y^{\pi(\overline{x})}(\overline{x})=\Phi_{B^{\star}}^{\pi(\overline{x})}(\overline{x})$, in particular the linear Poincar\'e map $P_Y^{\pi(\overline{x})}$ has only eigenvalues equal to $1$ or $-1$. Once we get this vector field we apply Lemma~\ref{linear2} to get a  new vector field $Z \in \mathfrak{X}_\mu^1(M)$, arbitrarily $C^1$-close to $Y$ and having a $Z^t$-invariant tubular neighbourhood of the orbit of $\overline{x}$, which  contradicts the fact that $Y$ is robustly transitive thus ending the proof of Proposition~\ref{VF}.

Therefore it remains to explain how we obtain the mentioned vector field $Y$.  Consider the arcs of trajectory $$\Gamma_j=\bigcup_{t \in[0,1]}X^t(X^j(\overline{x})), \text{ for } j \in \{0,..., [\pi(x)]-2\},$$
and 
$$\Gamma_{[\pi(x)]-1}=\bigcup_{t \in[0,1+\pi(\overline{x})-[\pi(\overline{x}]]}X^t(X^j(\overline{x})).$$
For  $j \in \{0,..., [\pi(x)]-1\}$ we write $\Phi_{B^{\star}}^1(X^j(\overline{x}))=P_X^1(X^j(\overline{x})) \circ A^j$, where $A^j={(P_X^1(X^j(\overline{x})))}^{-1} \circ \Phi_{B^{\star}}^1(X^j(\overline{x}))$, and consider a continuous one-parameter family $\tilde{A}_t^j$ defined by 
\begin{itemize}
\item $\tilde{A}_t^j=Id$ for $t \leq 0$;  
\item $\tilde{A}_t^j=A^j$ for $t \geq 1$ (for $t \geq  1+\pi(\overline{x})-[\pi(\overline{x})]$, if $j=[\pi(\overline{x})]-1$), and 
\item $\tilde{A}_t^j={(P_X^t(X^j(\overline{x})))}^{-1} \circ \Phi_{B^{\star}}^t(X^j(\overline{x}))$, for $0 \leq t \leq 1$ (or for $0 \leq t \leq 1+\pi(\overline{x})-[\pi(\overline{x})]$, if $j=[\pi(\overline{x})]-1$).
\end{itemize}
We fix small $\tilde{\delta}>0$ and consider $A_t^j$  a one-parameter linear family (see Section~\ref{pl}), arbitrarily close to $\tilde{A}_t^j$,  such that $A_t^j=Id$ for $t \leq \tilde{\delta}$ and $A_t^j=\tilde{A}_t^j$ for $t \geq 1$ (or $t \geq 1+\pi(\overline{x})-[\pi(\overline{x})]$, if $j=[\pi(\overline{x})]-1$). Now, we observe that $\| (A_t^j)^{\prime} (A_t^j)^{-1}\|$ is of order
$$\|((P_X^{t}(X^j(\overline{x})))^{-1} \circ \Phi_{B^{\star}}^t(X^j(\overline{x})))^{\prime} \circ ((P_X^{t}(X^j(\overline{x})))^{-1} \circ \Phi_{B^{\star}}^t(X^j(\overline{x})))^{-1}\|,$$
therefore of order
$$\mathcal{O}(\overline{x}) =\max_{y \in \Gamma(\overline{x}, \pi(\overline{x}))}\|B^{\star}_y - A^{\star}_y \|.$$
As we mention before $\overline{x} \in \Sigma^{\prime}$ can be chosen such that $\mathcal{O}(\overline{x})$ is arbitrarily small. Therefore, fixing $0 < \delta < \tilde{\delta}$ and small $r>0$, to each arc $\Gamma_j$ we apply Lemma~\ref{mpl} to get a new vector field $Y_j \in \mathfrak{X}_\mu^1(M)$, $\epsilon$-$C^1$-close to $Y$ and such that
\begin{enumerate}
\item $Y_j^t(\overline{x})=X_j^t(\overline{x})$, $\forall t \in \mathbb{R}$;
\item $P_{Y_j}^1(X^j(\overline{x}))=\Phi_{B^{\star}}^1(X^j(\overline{x}))$, and 
\item $Y_j|_{\mathcal{T}_j^c}=X|_{\mathcal{T}_j^c}$, where $\mathcal{T}_j=\mathcal{T}(X^j(\overline{x}), \tau, r, \delta)$ and $\tau = 1$ (or $\tau=1+\pi(\overline{x})-[\pi(\overline{x})]$  if $j=[\pi(\overline{x})]-1$).
\end{enumerate}
Finally, as $ \delta < \tilde{\delta}$,  by construction it follows that 
$$Y_j|_{\mathcal{T}_j \cap \mathcal{T}_{j+1}}=Y_{j+1}|_{\mathcal{T}_j \cap \mathcal{T}_{j+1}},\,\, \forall j \in \{0,1,..., [\pi(\overline{x})]-1\},$$ which, together with item $3.$ above, implies that these vector fields can be \emph{glued} to obtain a $C^1$ vector field $Y \in  \mathfrak{X}_\mu^1(M)$, $\epsilon$-$C^1$-close to $X$ and such that
$Y|_{\mathcal{T}_j}=Y_j|_{\mathcal{T}_j}$, for all $j \in \{0,1,..., [\pi(\overline{x})]-1\}$. Thus $Y^t(\overline{x})=X^t(\overline{x})$, for all $t \in \mathbb{R}$ and $P_Y^{\pi(\overline{x})}(\overline{x})=\Phi_{B^{\star}}^{\pi(\overline{x})}(\overline{x})$ as required. This ends the proof of Proposition~\ref{VF}.

\end{subsection}

\end{section}

\flushleft
{\bf M\'ario Bessa} \ \  (bessa@fc.up.pt)\\
CMUP, Rua do Campo Alegre, 687 \\ 4169-007 Porto \\ Portugal\\

\medskip

\flushleft
{\bf Jorge Rocha}  \ \  (jrocha@fc.up.pt)\\
Dep. Matem\'atica Pura-FCUP, Rua do Campo Alegre, 687 \\ 4169-007 Porto \\ Portugal\\

\end{document}